\documentclass[12pt,reqno]{amsart}
\usepackage{amsmath, amsthm, amssymb, stmaryrd}

\advance \topmargin by -\headheight
\advance \topmargin by -\headsep
     
\evensidemargin \oddsidemargin
\newtheorem{theorem}{Theorem}[section]
\newtheorem{lemma}[theorem]{Lemma}

\theoremstyle{definition}

\theoremstyle{remark}

\numberwithin{equation}{section}

\newcommand{\mmod}[1]{\,\,(\text{\rm mod}\,\,#1)}

\def\bfa{{\mathbf a}}

\def\bfh{{\mathbf h}}

\def\bfn{{\mathbf n}}

\def\bfu{{\mathbf u}}
\def\bfv{{\mathbf v}}

\def\bfx{{\mathbf x}}
\def\bfy{{\mathbf y}}
\def\bfz{{\mathbf z}}

\def\calI{{\mathcal I}}

 \def\Ktil{{\widetilde K}}

\def\Ktil{\widetilde K}

\def\dbN{{\mathbb N}}  
\def\dbR{{\mathbb R}}
\def\dbZ{{\mathbb Z}}\def\dbQ{{\mathbb Q}}

\def\grA{{\mathfrak A}}
\def\grB{{\mathfrak B}}

\def\grf{{\mathfrak f}}\def\grF{{\mathfrak F}}
\def\grg{{\mathfrak g}}

\def\grJ{{\mathfrak J}}
\def\grk{{\mathfrak k}} \def\grK{{\mathfrak K}}

\def\grm{{\mathfrak m}}\def\grM{{\mathfrak M}}
\def\grN{{\mathfrak N}}\def\grn{{\mathfrak n}}

\def\grS{{\mathfrak S}}\def\grP{{\mathfrak P}}
\def\grW{{\mathfrak W}}\def\grB{{\mathfrak B}}
\def\grK{{\mathfrak K}}\def\grp{{\mathfrak p}}

 \def\grX{{\mathfrak X}}

\def\grW{{\mathfrak W}}

\def\alp{{\alpha}} \def\bfalp{{\boldsymbol \alpha}} 
\def\bet{{\beta}}  \def\bfbet{{\boldsymbol \beta}}
\def\gam{{\gamma}} \def\Gam{{\Gamma}}
\def\bfgam{{\boldsymbol \gam}} 
\def\del{{\delta}} \def\Del{{\Delta}}

\def\tet{{\theta}}  
\def\Tet{{\Theta}} 

\def\lam{{\lambda}}

\def\eps{\varepsilon}

\def\le{\leqslant} \def\ge{\geqslant}

\def\d{{\,{\rm d}}}

\begin{document}
\title[Cubic Vinogradov system]{Subconvexity in the inhomogeneous\\ cubic Vinogradov 
system}
\author[Trevor D. Wooley]{Trevor D. Wooley}
\address{Department of Mathematics, Purdue University, 150 N. University Street, West 
Lafayette, IN 47907-2067, USA}
\email{twooley@purdue.edu}
\thanks{The author's work is supported by NSF grants DMS-2001549 and DMS-1854398.}
\date{}
\dedicatory{}
\begin{abstract}When $\bfh\in \dbZ^3$, denote by $B(X;\bfh)$ the number of integral 
solutions to the system
\[
\sum_{i=1}^6(x_i^j-y_i^j)=h_j\quad (1\le j\le 3),
\]
with $1\le x_i,y_i\le X$ $(1\le i\le 6)$. When $h_1\ne 0$ and appropriate local solubility 
conditions on $\bfh$ are met, we obtain an asymptotic formula for $B(X;\bfh)$, thereby 
establishing a subconvex local-global principle in the inhomogeneous cubic Vinogradov 
system. We obtain similar conclusions also when $h_1=0$, $h_2\ne 0$ and $X$ is 
sufficiently large in terms of $h_2$. Our arguments involve minor arc estimates going 
beyond square-root cancellation.
\end{abstract}
\maketitle

\section{Introduction} The application of the Hardy-Littlewood (circle) method in the 
asymptotic analysis of the number of integral solutions of a Diophantine system is, with few 
exceptions, limited to scenarios in which the number of variables is larger than twice the 
total degree of the system. This {\it convexity barrier} arises from the relative sizes of the 
putative main term, given by the product of local densities associated with the system, and 
the most optimistic bound anticipated for the error term, namely the square-root of the 
number of choices for the variables. Almost all of the exceptions to this rule are inherently 
linear \cite{GT2010} or quadratic \cite{Est1962, HB1996, Klo1927} in nature. There is work 
on pairs of diagonal cubic forms of special shape in $11$ or more variables \cite{BW2014}, 
and also an asymptotic formula for a special system consisting of one diagonal cubic and 
two linear equations in $10$ variables \cite{BW2019}. Recently, the author \cite{Woo2021} 
succeeded in breaking the convexity barrier for the Hilbert-Kamke problem of degree $k$, 
establishing an asymptotic formula for the number of solutions when the number of 
variables is at least $k(k+1)$. We turn our attention in this memoir to a related problem in 
which latent translation-dilation invariance obstructs the method of \cite{Woo2021}.\par

In order to describe our conclusions we must introduce some notation. Let $s$ be a 
positive number and $\bfh=(h_1,h_2,h_3)$ a triple of integers. When $X$ is a 
large real number, write
\begin{equation}\label{1.1}
f(\bfalp;X)=\sum_{1\le x\le X}e(\alp_1x+\alp_2x^2+\alp_3x^3),
\end{equation}
where $e(z)$ denotes $e^{2\pi iz}$. We consider the twisted mean value
\begin{equation}\label{1.2}
B_s(X;\bfh)=\int_{[0,1)^3}|f(\bfalp;X)|^{2s}e(-\bfalp \cdot \bfh)\d\bfalp ,
\end{equation}
in which we write $\bfalp\cdot\bfh$ for $\alp_1h_1+\alp_2h_2+\alp_3h_3$. Note that when 
$s\in \dbN$, it follows via orthogonality that the mean value $B_s(X;\bfh)$ counts the 
number of integral solutions of the system of equations
\begin{equation}\label{1.3}
\sum_{i=1}^s(x_i^j-y_i^j)=h_j\quad (1\le j\le 3),
\end{equation}
with $1\le x_i,y_i\le X$ $(1\le i\le s)$. This is the inhomogeneous cubic Vinogradov system 
of the title.\par

In order to describe asymptotic formulae associated with $B_s(X;\bfh)$, we introduce the 
generating functions
\begin{equation}\label{1.4}
I(\bfbet)=\int_0^1e(\bet_1\gam+\bet_2\gam^2+\bet_3\gam^3)\d\gam
\end{equation}
and
\begin{equation}\label{1.5}
S(q,\bfa)=\sum_{r=1}^q e_q(a_1r+a_2r^2+a_3r^3),
\end{equation}
in which $e_q(u)$ denotes $e^{2\pi i u/q}$. Next, put $n_j=h_jX^{-j}$ $(1\le j\le 3)$, and 
define
\begin{equation}\label{1.6}
\grJ_s(\bfh)=\int_{\dbR^3}|I(\bfbet )|^{2s}e(-\bfbet \cdot \bfn)\d\bfbet
\end{equation}
and
\begin{equation}\label{1.7}
\grS_s(\bfh)=\sum_{q=1}^\infty \sum_{\substack{1\le a_1,a_2,a_3\le q\\ 
(q,a_1,a_2,a_3)=1}}\left|q^{-1}S(q,\bfa)\right|^{2s}e_q(-\bfa \cdot \bfh).
\end{equation}
We note that the {\it singular integral} $\grJ_s(\bfh)$, and {\it singular series} 
$\grS_s(\bfh)$, are known to converge absolutely for $s>7/2$, and $s>4$, respectively 
(see \cite[Theorem 1]{Ark1984} or \cite[Theorem 3.7]{AKC2004}). 

\begin{theorem}\label{theorem1.1} Suppose that $\bfh\in \dbZ^3$ and $h_1\ne 0$. Let 
$s$ be a natural number with $s\ge 6$. Then whenever $X$ is sufficiently large in terms of 
$s$, one has
\begin{equation}\label{1.8}
B_s(X;\bfh)=\grJ_s(\bfh)\grS_s(\bfh)X^{2s-6}+o(X^{2s-6}),
\end{equation}
in which $0\le \grJ_s(\bfh)\ll 1$ and $0\le \grS_s(\bfh)\ll 1$. If the system (\ref{1.3}) 
possesses a non-singular real solution with positive coordinates, moreover, then 
$\grJ_s(\bfh)\gg 1$. Likewise, if the system (\ref{1.3}) possesses primitive non-singular 
$p$-adic solutions for each prime $p$, then $\grS_s(\bfh)\gg 1$.
\end{theorem}

Theorem \ref{theorem1.1} delivers a conclusion tantamount to a quantitative form of the 
Hasse principle for the system (\ref{1.3}) at the convexity barrier when $s=6$, and in that 
situation applies a minor arc estimate going beyond square-root cancellation. 
Thus, provided that the latter system admits appropriate non-singular solutions in every 
completion of $\dbQ$, one has an asymptotic formula of the shape $B_6(X;\bfh)\sim CX^6$ 
for a suitable positive number $C$. When $s\ge 7$, the conclusion of Theorem 
\ref{theorem1.1} is a routine consequence of the resolution \cite{Woo2016} of the cubic 
case of the main conjecture in Vinogradov's mean value theorem, as the reader may 
confirm by applying the methods of Arkhipov \cite{Ark1984}. Establishing such a conclusion 
when $s\le 6$, however, requires something of a breakthrough in order that the familiar 
square-root barrier in the circle method be surmounted.\par

Brandes and Hughes \cite{BH2021} have recently investigated the inhomogeneous case of 
Vinogradov's mean value theorem of degree $k$ in the subcritical regime. While this work 
shows, inter alia, that $B_s(X;\bfh)=o(X^s)$ for $s\le 5$, their methods fall short of 
providing conclusions for the critical exponent $s=6$ addressed by Theorem 
\ref{theorem1.1}. We remark that quantitative aspects of their conclusions have been 
sharpened in forthcoming work \cite{Woo2021b} of the author.\par

Theorem \ref{theorem1.1} addresses no scenario in which $h_1=0$. Although we are 
unable to obtain uniform conclusions in such a situation, we do obtain asymptotic formulae 
when $h_2\ne 0$ and $X$ is sufficiently large in terms of $h_2$.

\begin{theorem}\label{theorem1.2}
Let $s$ be a natural number with $s\ge 6$. Then the asymptotic formula (\ref{1.8}) holds 
when $h_2\ne 0$, and $X$ is sufficiently large in terms of $h_2$.
\end{theorem}

As a consequence of Fermat's theorem, the system (\ref{1.3}) has solutions only when 
$h_3\equiv h_1\mmod{3}$ and $h_3\equiv h_2\equiv h_1\mmod{2}$. Theorems 
\ref{theorem1.1} and \ref{theorem1.2} offer local-global principles incorporating such 
conditions. More significant is the proof of an asymptotic formula at the convexity barrier, 
wherein we have $12$ variables available and the sum of the degrees of the underlying 
equations is $6$. Hitherto, no such conclusion has been available for inhomogeneous 
Vinogradov systems of degree exceeding $2$. Unfortunately, our methods yield no 
conclusion analogous to Theorem \ref{theorem1.1} for Vinogradov systems of degree 
exceeding $3$.\par

We prove Theorem \ref{theorem1.1} by applying the circle method, a key ingredient in our 
argument being an estimate for the contribution of the minor arcs beyond square-root 
cancellation. This we achieve in \S\S2, 3 and 4 by adapting the author's work on the 
asymptotic formula in Waring's problem (see \cite{Woo2012b}). Ignoring for now the 
restriction to minor arcs, we observe that an integral shift $z$, with $1\le z\le X$, in every 
variable in the system (\ref{1.3}) generates the related system
\begin{equation}\label{1.9}
\left.{\begin{aligned}
\sum_{i=1}^s(u_i^3-v_i^3)&=h_3+3h_2z+3h_1z^2\\
\sum_{i=1}^s(u_i^2-v_i^2)&=h_2+2h_1z\\
\sum_{i=1}^s(u_i-v_i)&=h_1
\end{aligned}}\right\}
\end{equation}
in which $1\le u_i,v_i\le 2X$. There is now the potential for additional averaging using this 
new variable $z$. Were the polynomials on the right hand side of (\ref{1.9}) to have 
respective degrees $3$, $2$ and $1$, then an appropriate minor arc estimate would follow 
at once via Weyl's inequality. However, the degrees of the polynomials are too small 
for such a simple treatment to apply, and instead we must relate the system to auxiliary 
mixed systems. It is critical here that available Weyl estimates for cubic polynomials are 
relatively strong. Weaker estimates available for larger degrees are insufficient for our 
purposes. It is vital, moreover, that in these auxiliary mixed systems the degree of the 
polynomial $h_3+3h_2z+3h_1z^2$ be at least $2$. Indeed, were $h_1$ to be $0$, the 
resulting linear polynomial would offer insufficient scope for obtaining minor arc estimates of 
sufficient strength for application in the proof of Theorem \ref{theorem1.1}.\par

Having prepared the auxiliary lemma exploiting shifts in \S2, we prepare in \S3 the auxiliary 
mean value estimates required in \S4 for the derivation of our basic minor arc estimate 
breaking the classical convexity barrier. In \S5 we describe the Hardy-Littlewood dissection
required in our proof of Theorem \ref{theorem1.1}, and we reinterpret the conclusion of 
\S4 as a minor arc estimate in a form convenient for the application at hand. Some pruning 
manoeuvres convert this bound into two estimates more classically associated with minor 
arcs in \S6. From here, it remains in \S7 to analyse the contribution of the major arcs, and 
thereby we complete the proof of Theorem \ref{theorem1.1} drawing heavily on the work 
of Arkhipov \cite{Ark1984}. We devote \S8 to the discussion of the scenario in which 
$h_1=0$, and the proof of Theorem \ref{theorem1.2}. Here, at the cost of sacrificing 
uniformity with respect to $\bfh$ in our conclusions, it transpires that one may make use of 
recent work on small cap decouplings \cite{DGW2020} in order to salvage a viable analysis.
\par

Our basic parameter is $X$, a sufficiently large positive number. Whenever $\eps$ appears 
in a statement, either implicitly or explicitly, we assert that the statement holds for each 
$\eps>0$. Implicit constants in Vinogradov's notation $\ll$ and $\gg$ may depend on 
$\eps$. Vector notation in the form $\bfx=(x_1,\ldots,x_r)$ is used with the dimension $r$ 
depending on the course of the argument. Also, we write $(a_1,\ldots ,a_s)$ for the 
greatest common divisor of the integers $a_1,\ldots ,a_s$, ambiguity between ordered 
$s$-tuples and corresponding greatest common divisors being easily resolved by context. 
Finally, we write $\|\tet\|$ for $\min\{|\tet-m|:m\in \dbZ\}$.

\section{An auxiliary mean value estimate via shifts} Our starting point is a method applied 
in the proof of \cite[Theorem 2.1]{Woo2012b}, whereby the latent translation-dilation 
invariance of the system (\ref{1.3}) is applied to generate additional cancellation. Define 
$f(\bfalp;X)$ as in (\ref{1.1}). Then, when $\bfh\in \dbZ^3$ and $\grB\subseteq \dbR$ is 
measurable, we put
\begin{equation}\label{2.1}
I_s(\grB;X;\bfh)=\int_\grB\int_0^1\int_0^1|f(\bfalp;X)|^{2s}e(-\bfalp\cdot \bfh)\d\bfalp ,
\end{equation}
in which $\bfalp\cdot \bfh=\alp_1h_1+\alp_2h_2+\alp_3h_3$ and $\d\bfalp$ denotes 
$\d\alp_1\d\alp_2\d\alp_3$. Thus, in particular, we find from (\ref{1.2}) that 
$I_s([0,1);X;\bfh)=B_s(X;\bfh)$. We also make use of the auxiliary generating function
\begin{equation}\label{2.2}
g(\bfalp, \tet;X)=\sum_{1\le y\le X}e\left( y\tet+2h_1y\alp_2+(3h_2y+3h_1y^2)\alp_3
\right) .
\end{equation}

\begin{lemma}\label{lemma2.1} Suppose that $s\in \dbN$, $\bfh\in \dbZ^3$ and 
$\grB\subseteq \dbR$ is measurable. Then
\[
I_s(\grB;X;\bfh)\ll X^{-1}(\log X)^{2s}\sup_{\Gam\in [0,1)}\int_\grB \int_0^1\int_0^1 
|f(\bfalp;2X)^{2s}g(\bfalp,\Gam;X)|\d\bfalp .
\]
\end{lemma}

\begin{proof} For every integral shift $y$ with $1\le y\le X$, one has
\begin{equation}\label{2.3}
f(\bfalp;X)=\sum_{1+y\le x\le X+y}e\left( \alp_3(x-y)^3+\alp_2(x-y)^2+\alp_1(x-y)\right).
\end{equation}
Write
\begin{equation}\label{2.4}
\grf_y(\bfalp;\gam)=\sum_{1\le x\le 2X}e\left( \alp_3(x-y)^3+\alp_2(x-y)^2+(\alp_1+\gam)
(x-y)\right)
\end{equation}
and
\[
K(\gam)=\sum_{1\le z\le X}e(-\gam z).
\]
Then it follows from (\ref{2.3}) via orthogonality that when $1\le y\le X$, one has
\begin{equation}\label{2.5}
f(\bfalp;X)=\int_0^1\grf_y(\bfalp;\gam)K(\gam)\d\gam .
\end{equation}
Next we substitute (\ref{2.5}) into (\ref{2.1}). Define
\begin{equation}\label{2.6}
\grF_y(\bfalp;\bfgam)=\prod_{i=1}^s \grf_y(\bfalp;\gam_i)\grf_y(-\bfalp;-\gam_{s+i}),
\end{equation}
\[
\Ktil(\bfgam)=\prod_{i=1}^sK(\gam_i)K(-\gam_{s+i})
\]
and
\begin{equation}\label{2.7}
\calI(\bfgam;y;\bfh)=\int_\grB \int_0^1\int_0^1 \grF_y(\bfalp;\bfgam)e(-\bfalp \cdot \bfh)
\d\bfalp .
\end{equation}
Then, when $1\le y\le X$, we see that
\begin{equation}\label{2.8}
I_s(\grB;X;\bfh)=\int_{[0,1)^{2s}}\calI(\bfgam ;y;\bfh)\Ktil(\bfgam)\d\bfgam .
\end{equation}

\par By orthogonality, it is apparent from (\ref{2.6}) that
\begin{equation}\label{2.9}
\int_0^1\int_0^1 \grF_y(\bfalp;\bfgam)e(-\bfalp \cdot \bfh)\d\alp_1\d\alp_2 
=\sum_{1\le \bfx\le 2X}\Del(\alp_3,\bfgam;\bfh,y),
\end{equation}
where $\Del(\alp_3,\bfgam;\bfh,y)$ is equal to
\[
e\biggl( \sum_{i=1}^s \left( \alp_3\left( (x_i-y)^3-(x_{s+i}-y)^3\right)+\left( \gam_i(x_i-y)
-\gam_{s+i}(x_{s+i}-y)\right)\right)-\alp_3h_3\biggr) ,
\]
when
\begin{equation}\label{2.10}
\sum_{i=1}^s\left( (x_i-y)^j-(x_{s+i}-y)^j\right)=h_j\quad (j=1,2),
\end{equation}
and otherwise $\Del(\alp_3,\bfgam;\bfh,y)$ is equal to $0$.\par

\par By applying the binomial theorem within (\ref{2.10}), we obtain the relations
\begin{align*}
\sum_{i=1}^s(x_i-x_{s+i})&=h_1,\\
\sum_{i=1}^s(x_i^2-x_{s+i}^2)&=h_2+2yh_1,\\
\sum_{i=1}^s(x_i^3-x_{s+i}^3)&=3yh_2+3y^2h_1+\sum_{i=1}^s\left( (x_i-y)^3-
(x_{s+i}-y)^3\right) .
\end{align*}
Therefore, if we define $\Gam=\Gam(\bfgam)$ by taking
\[
\Gam(\bfgam)=\sum_{i=1}^s(\gam_i-\gam_{s+i}),
\]
and then write
\[
\grg_y(\bfalp;\bfh;\bfgam)=e\biggl( -\sum_{j=1}^3\alp_j\sum_{l=0}^{j-1}\binom{j}{l}
h_{j-l}y^l-y\Gam(\bfgam)\biggr),
\]
then we deduce from (\ref{2.6}) and (\ref{2.9}) that
\begin{equation}\label{2.11}
\int_0^1\int_0^1 \grF_y(\bfalp;\bfgam)e(-\bfalp \cdot \bfh)\d\alp_1\d\alp_2=
\int_0^1\int_0^1\grF_0(\bfalp;\bfgam)\grg_y(\bfalp;\bfh;\bfgam)\d\alp_1\d\alp_2 .
\end{equation}

\par Referring next to (\ref{2.8}), we see that when $X\in \dbN$ one obtains the relation
\[
I_s(\grB;X;\bfh)=X^{-1}\sum_{1\le y\le X}\int_{[0,1)^{2s}}
\calI(\bfgam;y;\bfh)\Ktil(\bfgam)\d\bfgam .
\]
Thus, we infer from (\ref{2.7}) and (\ref{2.11}) that
\begin{equation}\label{2.12}
I_s(\grB;X;\bfh)\ll X^{-1}\int_{[0,1)^{2s}}|H(\bfgam)\Ktil(\bfgam)|\d\bfgam ,
\end{equation}
where
\begin{equation}\label{2.13}
H(\bfgam)=\int_\grB \int_0^1\int_0^1 \grF_0(\bfalp;\bfgam)G(\bfalp;\bfh;\bfgam)\d\bfalp ,
\end{equation}
and
\begin{equation}\label{2.14}
G(\bfalp;\bfh;\bfgam)=\sum_{1\le y\le X}\grg_y(\bfalp;\bfh;\bfgam).
\end{equation}

\par We now aim to simplify the upper bound (\ref{2.12}). Observe first that by reference 
to (\ref{2.6}), it follows from the elementary inequality
\[
|z_1\cdots z_n|\le |z_1|^n+\ldots +|z_n|^n
\]
that
\[
|\grF_0(\bfalp;\bfgam)|\le \sum_{i=1}^{2s}|\grf_0(\bfalp;\gam_i)|^{2s}=
\sum_{i=1}^{2s}|\grf_0(\alp_3,\alp_2,\alp_1+\gam_i;0)|^{2s}.
\]
Also, from (\ref{2.2}) and (\ref{2.14}) we have
\begin{align*}
|G(\bfalp;\bfh;\bfgam)&=\biggl| \sum_{1\le y\le X}e\left( y\Gam(\bfgam)+2h_1y\alp_2+
(3h_2y+3h_1y^2)\alp_3\right) \biggr|\\
&=|g(\bfalp,\Gam(\bfgam);X)|.
\end{align*}
Therefore, since $|G(\bfalp;\bfh;\bfgam)|$ does not depend on $\alp_1$, it follows from 
(\ref{2.13}) via a change of variable that
\begin{equation}\label{2.16}
|H(\bfgam)|\le \int_\grB \int_0^1\int_0^1 |\grf_0(\bfalp;0)^{2s}g(\bfalp,\Gam(\bfgam);X)|
\d\bfalp .
\end{equation}

\par Define
\[
U_s(\grB)=\sup_{\Gam \in [0,1)}\int_\grB \int_0^1\int_0^1 |f(\bfalp;2X)^{2s}
g(\bfalp,\Gam;X)|\d\bfalp .
\]
Also, recall that
\[
\int_0^1|K(\gam)|\d\gam \ll \int_0^1\min\{ X,\|\gam\|^{-1}\}\d\gam \ll \log (2X),
\]
and note from (\ref{1.1}) and (\ref{2.4}) that $\grf_0(\bfalp;0)=f(\bfalp;2X)$. Then we 
deduce from (\ref{2.16}) that $|H(\bfgam)|\le U_s(\grB)$, and hence (\ref{2.12}) yields the 
bound
\begin{align*}
I_s(\grB;X;\bfh)&\ll X^{-1}U_s(\grB)\biggl( \int_0^1|K(\gam)|\d\gam \biggr)^{2s}\\ 
&\ll X^{-1}(\log (2X))^{2s}U_s(\grB).
\end{align*}
This completes the proof of the lemma.
\end{proof}

\section{Further auxiliary mean value estimates} We now prepare mean value estimates of 
use in bounding a minor arc contribution of utility in an application of the Hardy-Littlewood 
method. Recalling the exponential sum $g(\bfalp,\tet;X)$ defined in (\ref{2.2}), and writing 
$g(\bfalp;X)=g(\bfalp,0;X)$, these mixed mean values take the shape
\begin{equation}\label{3.1}
\Tet_m(X;\bfh)=\int_{[0,1)^3}|f(\bfalp;2X)^{2m}g(\bfalp;X)^6|\d\bfalp\quad (m\in\dbN) .
\end{equation}

\begin{lemma}\label{lemma3.1} When $\bfh\in \dbZ^3$ and $h_1\ne 0$, one has 
$\Tet_1(X;\bfh)\ll X^4\log (2X)$.
\end{lemma}

\begin{proof} By orthogonality, one has
\[
\int_0^1|f(\bfalp;2X)|^2\d\alp_1\le 2X.
\]
Then since $g(\bfalp;X)$ is independent of $\alp_1$, we deduce from (\ref{3.1}) that
\begin{equation}\label{3.2}
\Tet_1(X;\bfh)\le 2X\int_{[0,1)^2}|g(0,\alp_2,\alp_3;X)|^6\d\alp_2 \d\alp_3 .
\end{equation}
A second application of orthogonality reveals that the integral on the right hand side here 
counts the number of integral solutions $T_0(X)$ of the system
\begin{align*}
3h_1\sum_{i=1}^3(x_i^2-y_i^2)+3h_2\sum_{i=1}^3(x_i-y_i)&=0,\\
2h_1\sum_{i=1}^3(x_i-y_i)&=0,
\end{align*}
with $1\le x_i,y_i\le X$ $(1\le i\le 3)$. Since, by hypothesis, one has $h_1\ne 0$, we see 
that $T_0(X)$ counts the integral solutions of the Vinogradov system of equations
\[
\sum_{i=1}^3(x_i^j-y_i^j)=0\quad (j=1,2),
\]
with the same conditions on $\bfx$ and $\bfy$. Thus $T_0(X)\ll X^3\log (2X)$ (a precise 
asymptotic formula can be found in \cite{BB2010}), and the conclusion of the lemma 
follows by substituting this upper bound into (\ref{3.2}).
\end{proof}

We next consider mean values in which the multiplicity of the generating functions 
$f(\bfalp;2X)$ is increased by appealing to the Hardy-Littlewood method. With this goal in 
mind, we introduce a Hardy-Littlewood dissection. When $Q$ is a real parameter with 
$1\le Q\le X$, we define the set of major arcs $\grM(Q)$ to be the union of the arcs
\[
\grM(q,a)=\{ \alp\in [0,1):|q\alp -a|\le QX^{-3}\},
\]
with $0\le a\le q\le Q$ and $(a,q)=1$. We then define the complementary set of minor arcs 
$\grm(Q)=[0,1)\setminus \grM(Q)$.\par

We begin with a familiar auxiliary bound for $f(\bfalp;2X)$. In this context, it is useful to 
define the function $\Psi(\alp)$ for $\alp\in [0,1)$ by putting
\[
\Psi(\alp)=(q+X^3|q\alp-a|)^{-1},
\]
when $\alp \in \grM(q,a)\subseteq \grM(X)$, and otherwise by taking $\Psi(\alp)=0$.\par

\begin{lemma}\label{lemma3.2} One has 
$f(\bfalp;2X)^4\ll X^{3+\eps}+X^{4+\eps}\Psi(\alp_3)$.
\end{lemma}

\begin{proof} Suppose that $\alp_3\in \dbR$, and $a\in \dbZ$ and $q\in \dbN$ satisfy 
$(a,q)=1$ and $|\alp_3-a/q|\le q^{-2}$. Then from Weyl's inequality (see 
\cite[Lemma 2.4]{Vau1997}), we have
\begin{equation}\label{3.3}
|f(\bfalp;2X)|\ll X^{1+\eps}(q^{-1}+X^{-1}+qX^{-3})^{1/4}. 
\end{equation}
Hence, by a standard transference principle (see \cite[Lemma 14.1]{Woo2015a}), whenever 
$a\in \dbZ$ and $q\in \dbN$ satisfy $(a,q)=1$, one has
\begin{equation}\label{3.4}
|f(\bfalp;2X)|\ll X^{1+\eps}(\lam^{-1}+X^{-1}+\lam X^{-3})^{1/4},
\end{equation}
where $\lam=q+X^3|q\alp_3-a|$.\par

When $\alp_3\in [0,1)$, an application of Dirichlet's 
approximation theorem shows that there exist $a\in \dbZ$ and $q\in \dbN$ with 
$0\le a\le q\le X^2$, $(a,q)=1$ and $|q\alp_3-a|\le X^{-2}$. Thus 
$\lam=q+X^3|q\alp_3-a|\ll X^2$. Note that when $\alp_3\in \grm(X)$ we must have 
$\lam\ge q>X$, whilst for $\alp_3\in \grM(X)$ one has $\lam^{-1}=\Psi(\alp_3)$. Then in 
any case we find from (\ref{3.4}) that
\[
|f(\bfalp;2X)|^4\ll X^{3+\eps}+X^{4+\eps}\lam^{-1}\ll X^{3+\eps}
+X^{4+\eps}\Psi(\alp_3),
\]
and the conclusion of the lemma follows.
\end{proof}

\begin{lemma}\label{lemma3.3}
Suppose that $\bfh\in \dbZ^3$ and $h_1\ne 0$. Then one has
\[
\Tet_3(X;\bfh)\ll X^{7+\eps}\quad \text{and}\quad \Tet_5(X;\bfh)\ll X^{10+\eps}.
\]
\end{lemma}

\begin{proof} By applying Lemma \ref{lemma3.2} to (\ref{3.1}), one obtains
\begin{equation}\label{3.5}
\Tet_3(X;\bfh)\ll X^{3+\eps}\Tet_1(X;\bfh)+X^{4+\eps}T_1,
\end{equation}
where
\[
T_1=\int_{[0,1)^3}\Psi(\alp_3)|f(\bfalp;2X)^2g(\bfalp;X)^6|\d\bfalp .
\]
Moreover, as a consequence of \cite[Lemma 2]{Bru1988}, we have
\begin{equation}\label{3.6}
T_1\ll X^{\eps-3}(X\Tet_1(X;\bfh)+T_2),
\end{equation}
where
\[
T_2=\int_{[0,1)^2}|f(\alp_1,\alp_2,0;2X)^2g(\alp_1,\alp_2,0;X)^6|\d\alp_1\d\alp_2 .
\]
By orthogonality, the mean value $T_2$ counts the integral solutions of the simultaneous 
equations
\begin{align*}
x_1^2-x_2^2&=2h_1(y_1+y_2+y_3-y_4-y_5-y_6),\\
x_1-x_2&=0,
\end{align*}
with $1\le x_1,x_2\le 2X$ and $1\le y_i\le X$ $(1\le i\le 6)$. Plainly, in any such solution one 
has $x_1=x_2$, and so there are at most $O(X)$ possible choices for $x_1$ and $x_2$. 
Meanwhile, given $y_1,\ldots ,y_5$, the variable $y_6$ is determined uniquely from the 
first of these equations, so there are $O(X^5)$ possible choices for $\bfy$. We therefore 
see that $T_2=O(X^6)$, and hence (\ref{3.6}) delivers the bound
\[
T_1\ll X^{\eps-2}\Tet_1(X;\bfh)+X^{3+\eps}.
\]
The first bound of the lemma follows by substituting this estimate into (\ref{3.5}), noting 
the bound $\Tet_1(X;\bfh)\ll X^{4+\eps}$ available from Lemma \ref{lemma3.1}.\par

The second bound of the lemma is obtained by applying Lemma \ref{lemma3.2} to 
(\ref{3.1}) again, yielding
\begin{equation}\label{3.7}
\Tet_5(X;\bfh)\ll X^{3+\eps}\Tet_3(X;\bfh)+X^{4+\eps}T_3,
\end{equation}
where
\[
T_3=\int_{[0,1)^3}\Psi(\alp_3)|f(\bfalp;2X)^6g(\bfalp ;X)^6|\d\bfalp .
\]
Again utilising \cite[Lemma 2]{Bru1988}, we deduce that
\begin{equation}\label{3.8}
T_3\ll X^{\eps-3}(X\Tet_3(X;\bfh)+T_4),
\end{equation}
where
\[
T_4=\int_{[0,1)^2}|f(\alp_1,\alp_2,0;2X)^6g(\alp_2,\alp_2,0;X)^6|\d \alp_1\d\alp_2 .
\]
By applying the trivial estimate $|g(\alp_1,\alp_2,0;X)|=O(X)$, we see that
\[
T_4\ll X^6\int_{[0,1)^2}|f(\alp_1,\alp_2,0;2X)|^6\d\alp_1\d\alp_2\ll X^6\cdot X^{3+\eps}.
\]
Here, applying orthogonality, we recognised that the last integral is equal to our 
acquaintance $T_0(2X)$ introduced in the proof of Lemma \ref{lemma3.1}, and shown 
therein to be $O(X^{3+\eps})$. We thus deduce from (\ref{3.8}) that
\[
T_3\ll X^{\eps-2}\Tet_3(X;\bfh)+X^{6+\eps}.
\]
The second bound of the lemma follows by substituting this estimate into (\ref{3.7}), noting 
the first bound $\Tet_3(X;\bfh)\ll X^{7+\eps}$ already obtained.
\end{proof}

We convert the second bound of Lemma \ref{lemma3.3} into one suitable for later use. In 
this context, it is convenient to introduce the mean value 
\[
V(\Gam;X;\bfh)=\int_{[0,1)^3}|f(\bfalp;2X)^{10}g(\bfalp,\Gam;X)^6|\d\bfalp .
\]

\begin{lemma}\label{lemma3.4}
Suppose that $\bfh\in \dbZ^3$ and $h_1\ne 0$. Then one has
\[
\sup_{\Gam\in [0,1)}V(\Gam;X;\bfh)\ll X^{10+\eps}.
\]
\end{lemma}

\begin{proof} By orthogonality, the mean value $V(\Gam;X;\bfh)$ counts the integral 
solutions of the system
\begin{align*}
\sum_{i=1}^5(x_i^3-y_i^3)&=3h_2\sum_{j=1}^3(u_j-v_j)+3h_1\sum_{j=1}^3
(u_j^2-v_j^2)\\
\sum_{i=1}^5(x_i^2-y_i^2)&=2h_1\sum_{j=1}^3(u_j-v_j)\\
\sum_{i=1}^5(x_i-y_i)&=0,
\end{align*}
with $1\le x_i,y_i\le 2X$ $(1\le i\le 5)$ and $1\le u_j,v_j\le X$ $(1\le j\le 3)$, and with 
each solution $\bfx,\bfy,\bfu,\bfv$ being counted with weight
\[
e(-\Gam (u_1+u_2+u_3-v_1-v_2-v_3)).
\]
Since the latter weight is unimodular, we obtain an upper bound for $V(\Gam;X;\bfh)$ by 
replacing that weight with $1$, or equivalently, by setting $\Gam$ to be $0$. Thus, by 
reference to (\ref{3.1}), we conclude that
\[
\sup_{\Gam\in [0,1)}V(\Gam;X;\bfh)\le V(0;X;\bfh)=\Tet_5(X;\bfh).
\]
The conclusion of the lemma is therefore immediate from Lemma \ref{lemma3.3}.
\end{proof}

\section{A first minor arc bound} Before announcing our first estimate of minor arc type, we 
recall a standard consequence of Weyl's inequality. Recall the set of minor arcs $\grm(Q)$ 
defined in the preamble to Lemma \ref{lemma3.2}, and suppose that $\alp_3\in \grm(Q)$. 
By Dirichlet's approximation theorem, there exist $a\in \dbZ$ and $q\in \dbN$ with 
$0\le a\le q\le Q^{-1}X^3$, $(a,q)=1$ and $|q\alp_3-a|\le QX^{-3}$. Since 
$\alp_3\in \grm(Q)$ one has $q>Q$, and thus we deduce from Weyl's inequality (\ref{3.3}) 
that
\begin{equation}\label{4.1}
\sup_{\alp_3\in \grm(Q)}\sup_{(\alp_1,\alp_2)\in [0,1)^2}|f(\bfalp;2X)|\ll 
X^{1+\eps}Q^{-1/4}.
\end{equation}

\begin{lemma}\label{lemma4.1} Let $s$ be a natural number with $s\ge 6$, and put 
$\del=2s-35/3$. Then whenever $\bfh\in \dbZ^3$ and $h_1\ne 0$, one has
\[
I_s(\grm(Q);X;\bfh)\ll X^{2s-6+\eps}Q^{-\del/4}.
\]
\end{lemma}

\begin{proof} We find from Lemma \ref{lemma2.1} that
\begin{equation}\label{4.2}
I_s(\grm(Q);X;\bfh)\ll X^{\eps-1}\sup_{\Gam\in [0,1)}W_s(\Gam;X;\bfh),
\end{equation}
where
\[
W_s(\Gam;X;\bfh)=\int_{\grm(Q)}\int_0^1\int_0^1|f(\bfalp;2X)^{2s}g(\bfalp,\Gam;X)|
\d\bfalp .
\]
An application of H\"older's inequality reveals that whenever $s\ge 6$, one has
\begin{equation}\label{4.3}
W_s(\Gam;X;\bfh)\le \biggl( \sup_{\alp_3\in \grm(Q)}\sup_{(\alp_1,\alp_2)\in [0,1)^2}|
f(\bfalp;2X)|\biggr)^{2s-35/3}U_1^{5/6}U_2^{1/6},
\end{equation}
in which we write
\[
U_1=\int_{[0,1)^3}|f(\bfalp;2X)|^{12}\d\bfalp \quad \text{and}\quad 
U_2=\int_{[0,1)^3}|f(\bfalp;2X)^{10}g(\bfalp,\Gam;X)^6|\d\bfalp.
\]
The cubic case of the main conjecture in Vinogradov's mean value theorem established by 
the author \cite{Woo2016} shows that $U_1\ll X^{6+\eps}$. Meanwhile, the bound 
$U_2\ll X^{10+\eps}$ is confirmed in Lemma \ref{lemma3.4}. By substituting these bounds 
together with (\ref{4.1}) into (\ref{4.3}), we obtain the estimate
\[
W_s(\Gam;X;\bfh)\ll X^\eps \left( XQ^{-1/4}\right)^{2s-35/3}\left( X^6\right)^{5/6}
\left( X^{10}\right)^{1/6}\ll X^{2s-5+\eps}Q^{-\del/4}.
\]
The conclusion of the lemma follows by substituting this bound into (\ref{4.2}).
\end{proof}

\section{The Hardy-Littlewood dissection} Our application of the Hardy-Littlewood method 
follows the strategy pursued in our recent work on the Hilbert-Kamke problem (see 
\cite{Woo2021}), though equipped in this instance with the minor arc estimate prepared in 
\S4. We begin our discussion by introducing a close relative of the mean value 
$I_s(\grB;X;\bfh)$ introduced in (\ref{2.1}). Thus, when $\grA\subseteq [0,1)^3$ is 
measurable, we define the mean value $T_s(\grA;\bfh)=T_s(\grA;X;\bfh)$ by putting
\begin{equation}\label{5.1}
T_s(\grA;X;\bfh)=\int_\grA |f(\bfalp;X)|^{2s}e(-\bfalp \cdot \bfh)\d\bfalp .
\end{equation}

\par We require an appropriate Hardy-Littlewood dissection of the unit cube $[0,1)^3$ into 
major and minor arcs. When $Z$ is a real parameter with $1\le Z\le X$, we define the set 
of major arcs $\grK(Z)$ to be the union of the arcs
\[
\grK(q,\bfa;Z)=\{\bfalp \in [0,1)^3:\text{$|\alp_j-a_j/q|\le ZX^{-j}$ $(1\le j\le 3)$}\},
\]
with $1\le q\le Z$, $0\le a_j\le q$ $(1\le j\le 3)$ and $(q,a_1,a_2,a_3)=1$. We then define 
the complementary set of minor arcs $\grk(Z)=[0,1)^3\setminus \grK(Z)$.\par

We have already defined the one-dimensional Hardy-Littlewood dissection of $[0,1)$ into the 
sets of arcs $\grM(Q)$ and $\grm(Q)$. We now fix $L=X^{1/72}$ and $Q=L^3$, and we 
define intermediate sets of $3$-dimensional arcs $\grN=\grK(Q^2)$ and $\grn=\grk(Q^2)$. 
We also need a narrow set of major arcs $\grP=\grK(L)$ and a corresponding set of minor 
arcs $\grp=\grk(L)$. It is convenient, in this context, to write $\grP(q,\bfa)=\grK(q,\bfa;L)$. 
As is easily verified, one has $\grP\subseteq [0,1)^2\times\grM$. Hence, the set of points 
$(\alp_1,\alp_2,\alp_3)$ lying in $[0,1)^3$ may be partitioned into the four disjoint subsets
\begin{align*}
\grW_1&=[0,1)^2\times \grm,\\
\grW_2&=\left( [0,1)^2\times \grM\right)\cap \grn,\\
\grW_3&=\left( [0,1)^2\times \grM\right) \cap (\grN\setminus \grP),\\
\grW_4&=\grP.
\end{align*}
Thus, on comparing (\ref{1.2}) and (\ref{5.1}), we find that
\begin{equation}\label{5.2}
B_s(X;\bfh)=T_s([0,1)^3;X;\bfh)=\sum_{i=1}^4T_s(\grW_i;X;\bfh).
\end{equation}

\par Our work in \S\S2, 3 and 4 bounds $T_s(\grW_1;\bfh)$.

\begin{lemma}\label{lemma5.1} Let $s$ be a positive integer with $s\ge 6$, and put 
$\del=2s-35/3$. Then, whenever $\bfh\in \dbZ^3$ and $h_1\ne 0$, one has
\[
T_s(\grW_1;\bfh)\ll X^{2s-6-\del/100}.
\]
\end{lemma}

\begin{proof} By substituting $Q=X^{1/24}$ into Lemma \ref{lemma4.1}, we find that
\[
T_s(\grW_1;\bfh)=I_s(\grm(Q);X;\bfh)\ll X^{2s-6+\eps}\cdot X^{-\del/96},
\]
and the conclusion of the lemma follows at once.
\end{proof}

\section{Further minor arc estimates} Our analysis of the sets of arcs $\grW_2$ and 
$\grW_3$ within (\ref{5.2}) involves standard tools from the theory of Vinogradov's mean 
value theorem. We begin by recording an estimate of Weyl type for the exponential sum 
$f(\bfalp;X)$.

\begin{lemma}\label{lemma6.1} One has
\[
\sup_{\bfalp \in \grn}|f(\bfalp;X)|\ll X^{1-1/54}\quad \text{and}\quad 
\sup_{\bfalp\in \grp}|f(\bfalp;X)|\ll X^{1-1/324}.
\]
\end{lemma}

\begin{proof} This is the case $k=3$ of \cite[Lemma 4.1]{Woo2021}.
\end{proof}

\begin{lemma}\label{lemma6.2} Suppose that $\bfh\in \dbZ^3$ and $2s\ge 11$. Then one 
has
\[
T_s(\grW_2;\bfh)\ll X^{2s-6-1/110}.
\]
\end{lemma}

\begin{proof} By applying the triangle inequality to (\ref{5.1}), we find that
\begin{equation}\label{6.1}
T_s(\grW_2;\bfh)\ll \Bigl( \sup_{\bfalp\in \grn}|f(\bfalp;X)|\Bigr)^{2s-6}\int_\grM 
\int_0^1\int_0^1 |f(\bfalp;X)|^6\d\bfalp .
\end{equation}
Here, by orthogonality, the inner mean value
\[
\int_{[0,1)^2}|f(\bfalp;X)|^6\d\alp_1\d\alp_2
\]
counts the integral solutions of the system of equations
\[
\sum_{i=1}^3(x_i^j-y_i^j)=0\quad (j=1,2),
\]
with $1\le x_i,y_i\le X$ $(1\le i\le 3)$, and with each solution $\bfx,\bfy$ being counted with 
the unimodular weight
\[
e\biggl( \alp_3\sum_{i=1}^3(x_i^3-y_i^3)\biggr) .
\]
Making use of the familiar bound in the quadratic case of Vinogradov's mean value theorem, 
and observing that $\text{mes}(\grM)\ll Q^2X^{-3}$, we thus conclude that
\begin{equation}\label{6.2}
\int_\grM \int_0^1\int_0^1 |f(\bfalp;X)|^6\d\bfalp \ll X^{3+\eps}\text{mes}(\grM)\ll 
Q^2X^\eps .
\end{equation}

\par Substituting the bound (\ref{6.2}) into (\ref{6.1}), and invoking Lemma 
\ref{lemma6.1}, we obtain
\[
T_s(\grW_2;\bfh)\ll (X^{1-1/54})^{2s-6}Q^2X^\eps \ll X^{2s-6}(X^{\eps-5/54}Q^2).
\]
Since $Q=X^{1/24}$, the conclusion of the lemma follows at once.
\end{proof}

The analysis of the set of arcs $\grW_3$ is accomplished via the standard literature.

\begin{lemma}\label{lemma6.3} When $u>8$, one has
\[
\int_\grN |f(\bfalp ;X)|^u\d\bfalp \ll_u X^{u-6}.
\]
\end{lemma}

\begin{proof} This is essentially \cite[Lemma 7.1]{Woo2017}, following by the methods of 
\cite{AKC2004}.
\end{proof}

We may now announce our estimate for $T_s(\grW_3;\bfh)$.

\begin{lemma}\label{lemma6.4} Suppose that $\bfh\in \dbZ^3$ and $s\ge 5$. Then one 
has
\[
T_s(\grW_3;\bfh)\ll X^{2s-6-1/324}.
\]
\end{lemma}

\begin{proof} Since $\grW_3\subseteq \grN\setminus \grP$, we have
\[
\sup_{\bfalp\in \grW_3}|f(\bfalp;X)|\le \sup_{\bfalp\in \grp}|f(\bfalp;X)|.
\]
Thus, by the triangle inequality,
\[
T_s(\grW_3;\bfh)\ll X^{2s-10}\Bigl( \sup_{\bfalp\in \grp}|f(\bfalp;X)|\Bigr) 
\int_\grN |f(\bfalp;X)|^9\d\bfalp .
\]
Then as a consequence of Lemmata \ref{lemma6.1} and \ref{lemma6.3}, we obtain the 
bound
\[
T_s(\grW_3;\bfh)\ll X^{2s-10}\cdot X^{1-1/324}\cdot X^3\ll X^{2s-6-1/324}.
\]
This completes the proof of the lemma.
\end{proof}

\section{The major arc contribution} By substituting the estimates supplied by Lemmata 
\ref{lemma5.1}, \ref{lemma6.2} and \ref{lemma6.4} into (\ref{5.2}), noting also that 
$\grW_4=\grP$, we find that 
whenever $s\ge 6$, one has
\begin{equation}\label{7.1}
B_s(X;\bfh)=T_s(\grP;\bfh)+o(X^{2s-6}).
\end{equation}
In this section we analyse the major arc contribution $T_s(\grP;\bfh)$. This is routine, 
though the small number of available variables requires appropriate recourse to the 
literature.\par

Recall the notation (\ref{1.4}) and (\ref{1.5}). When $\bfalp\in \grP(q,\bfa)\subseteq \grP$, 
put
\[
V(\bfalp;q,\bfa)=q^{-1}S(q,\bfa)I(\bfalp-\bfa/q;X),
\]
where we write
\[
I(\bfbet;X)=\int_0^Xe(\bet_1\gam+\bet_2\gam^2+\bet_3\gam^3)\d\gam 
=XI\left(\bet_1 X,\bet_2 X^2,\bet_3 X^3\right).
\]
We then define the function $V(\bfalp)$ to be $V(\bfalp;q,\bfa)$ when 
$\bfalp\in \grP(q,\bfa)\subseteq \grP$, and to be zero otherwise. It now follows from 
\cite[Theorem 7.2]{Vau1997} that when $\bfalp\in \grP(q,\bfa)\subseteq \grP$, one has
\begin{align*}
f(\bfalp;X)-V(\bfalp)&\ll q+X|q\alp_1-a_1|+X^2|q\alp_2-a_2|+X^3|q\alp_3-a_3|\\
&\ll L^2.
\end{align*}
Thus, uniformly in $\bfalp\in \grP$, we have
\[
|f(\bfalp;X)|^{2s}-|V(\bfalp)|^{2s}\ll X^{2s-1}L^2.
\]
Since $\text{mes}(\grP)\ll L^7X^{-6}$, we deduce from (\ref{5.1}) that
\begin{equation}\label{7.2}
T_s(\grP;\bfh)=\int_\grP|V(\bfalp)|^{2s}e(-\bfalp \cdot \bfh)\d\bfalp +O(L^9X^{2s-7}).
\end{equation}

\par Next, applying the definition of $\grP$ in the familiar manner, we see that
\begin{equation}\label{7.3}
\int_\grP |V(\bfalp)|^{2s}e(-\bfalp \cdot \bfh)\d\bfalp =\grS_s(X;\bfh)\grJ_s(X;\bfh),
\end{equation}
where
\[
\grJ_s(X;\bfh)=\int_\grX |I(\bfbet;X)|^{2s}e(-\bfbet \cdot \bfh)\d\bfbet 
\]
and
\[
\grS_s(X;\bfh)=\sum_{1\le q\le L}\sum_{\substack{1\le \bfa\le q\\ (q,a_1,a_2,a_3)=1}}
q^{-2s}|S(q,\bfa)|^{2s}e_q(-\bfa\cdot \bfh),
\]
in which we write
\[
\grX=[-LX^{-1},LX^{-1}]\times [-LX^{-2},LX^{-2}]\times [-LX^{-3},LX^{-3}].
\]

\par The singular integral (\ref{1.6}) converges absolutely for $s>7/2$ (see 
\cite[Theorem 1.3]{AKC2004} or \cite[Theorem 1]{Ark1984}), and moreover 
\cite[Theorem 7.3]{Vau1997} supplies the bound
\[
I(\bfbet;X)\ll X(1+|\bet_1|X+|\bet_2|X^2+|\bet_3|X^3)^{-1/3}.
\]
Then we infer via two changes of variable that
\begin{align}
\grJ_s(X;\bfh)&=X^{2s-6}\int_{\dbR^3}|I(\bfbet)|^{2s}e\Bigl( -\bet_1\frac{h_1}{X}
-\bet_2\frac{h_2}{X^2}-\bet_3\frac{h_3}{X^3}\Bigr) \d\bfbet +o(X^{2s-6})\notag \\
&=\left( \grJ_s(\bfh)+o(1)\right) X^{2s-6}\ll X^{2s-6}.\label{7.4}
\end{align}
Similarly, the singular series (\ref{1.7}) converges absolutely for $s>4$ (see 
\cite[Theorem 2.4]{AKC2004} or \cite[Theorem 1]{Ark1984}), and in addition 
\cite[Theorem 7.1]{Vau1997} shows that when $(q,a_1,a_2,a_3)=1$, one has 
$S(q,\bfa)\ll q^{2/3+\eps}$. Thus, it follows that
\begin{equation}\label{7.5}
\grS_s(X;\bfh)=\grS_s(\bfh)+o(1)\ll 1.
\end{equation}

\par By substituting (\ref{7.4}) and (\ref{7.5}) into (\ref{7.3}), and thence into (\ref{7.2}), 
we obtain 
\begin{align*}
T_s(\grP;\bfh)&=\left( \grS_s(\bfh)+o(1)\right) \left( \grJ_s(\bfh)+o(1)\right) 
X^{2s-6}+o(X^{2s-6})\\
&=\grS_s(\bfh)\grJ_s(\bfh)X^{2s-6}+o(X^{2s-6}).
\end{align*}
By substituting this relation into (\ref{7.1}), we conclude that
\[
B_s(X;\bfh)=\grS_s(\bfh)\grJ_s(\bfh)X^{2s-6}+o(X^{2s-6}).
\]

\par We note that the absolute convergence of the integral $\grJ_s(\bfh)$ and of the series 
$\grS_s(\bfh)$ shows, via familiar technology from the circle method, that
\[
0\le \grJ_s(\bfh)\ll 1\quad \text{and}\quad 0\le \grS_s(\bfh)\ll 1.
\]
This standard technology also shows that the singular series may be written in the form
\[
\grS_s(\bfh)=\prod_p\varpi_p(s,\bfh),
\]
where for each prime number $p$, the $p$-adic density $\varpi_p(s,\bfh)$ is defined by
\[
\varpi_p(s,\bfh)=\sum_{h=0}^\infty \sum_{\substack{1\le \bfa\le p^h\\ 
(p,a_1,a_2,a_3)=1}}p^{-2sh}|S(p^h,\bfa)|^{2s}e_{p^h}(-\bfa \cdot \bfh).
\]
The positivity of $\grJ_s(\bfh)$ and $\grS_s(\bfh)$ corresponds to the existence of 
non-singular real and $p$-adic solutions to the system (\ref{1.3}). Granted the 
existence of primitive such solutions, the standard theory shows that $\grJ_s(\bfh)\gg 1$ 
and $\grS_s(\bfh)\gg 1$. This confirms the conclusion of Theorem \ref{theorem1.1}.

\section{Non-uniform conclusions: the proof of Theorem 1.2} In our proof of Theorem 
\ref{theorem1.2}, we abandon the uniformity in $\bfh$ implicit in the error term of Theorem
\ref{theorem1.1}, though now we require only that $h_2\ne 0$. The case $h_1\ne 0$ 
having already been handled in Theorem \ref{theorem1.1}, we assume that $h_1=0$ and 
$h_2\ne 0$. In such circumstances, it now follows from (\ref{2.2}) that
\[
g(\bfalp;X)=\sum_{1\le y\le X}e(3h_2y\alp_3).
\]

\par We begin by deriving an analogue of Lemma \ref{lemma3.3}.

\begin{lemma}\label{lemma8.1} Suppose that $\bfh\in \dbZ^3$ and $h_1=0$, $h_2\ne 0$. 
Then one has
\begin{equation}\label{8.1}
\int_{[0,1)^3}|f(\bfalp;2X)^{10}g(\bfalp;X)^2|\d\bfalp \ll_{h_2}X^{6+\eps}.
\end{equation}
\end{lemma}

\begin{proof} By orthogonality, the mean value in (\ref{8.1}) counts the integral solutions of 
the system of equations
\begin{align}
\sum_{i=1}^5(x_i^3-y_i^3)&=3h_2(z_1-z_2),\notag\\
\sum_{i=1}^5(x_i^2-y_i^2)&=0,\label{8.2}\\
\sum_{i=1}^5(x_i-y_i)&=0\notag,
\end{align}
with $1\le x_i,y_i\le 2X$ $(1\le i\le 5)$ and $1\le z_1,z_2\le X$. Thus, we find that
\begin{equation}\label{8.3}
\int_{[0,1)^3}|f(\bfalp;2X)^{10}g(\bfalp;X)^2|\d\bfalp \ll X U_1,
\end{equation}
where $U_1$ counts the number of integral solutions of the system
\begin{align*}
\biggl| \sum_{i=1}^5(x_i^3-y_i^3)\biggr|&<3|h_2|X,\\
\sum_{i=1}^5(x_i^2-y_i^2)&=0,\\
\sum_{i=1}^5(x_i-y_i)&=0,
\end{align*}
with $1\le x_i,y_i\le 2X$ $(1\le i\le 5)$.\par

A standard argument (see \cite[Lemma 2.1]{Wat1989}) shows from here that
\[
U_1\ll |h_2|X\int_{-1/X}^{1/X}\int_0^1\int_0^1|f(\bfalp;2X)|^{10}\d\bfalp .
\]
As a direct consequence of \cite[Theorem 3.3]{DGW2020}, however, we have the bound
\begin{equation}\label{8.4}
\int_{-1/X}^{1/X}\int_0^1\int_0^1|f(\bfalp;2X)|^{10}\d\bfalp \ll X^{4+\eps}.
\end{equation}
Thus we deduce that $U_1\ll |h_2|X^{5+\eps}$, and the upper bound of the lemma 
follows at once from (\ref{8.3}).
\end{proof}

By substituting the estimate (\ref{8.1}) within the argument of the proof of Lemma 
\ref{lemma3.4}, we readily deduce the bound contained in the following lemma.

\begin{lemma}\label{lemma8.2} Suppose that $\bfh\in \dbZ^3$ and $h_1=0$, $h_2\ne 0$. 
Then one has
\[
\sup_{\Gam\in [0,1)}\int_{[0,1)^3}|f(\bfalp;2X)^{10}g(\bfalp, \Gam;X)^2|\d\bfalp 
\ll_{h_2} X^{6+\eps}.
\]
\end{lemma}

\begin{proof} We may proceed as in the proof of Lemma \ref{lemma3.4}, adopting the 
notation therein. Thus, the mean value
\[
\int_{[0,1)^3}|f(\bfalp;2X)^{10}g(\bfalp, \Gam;X)^2|\d\bfalp 
\]
counts the integral solutions of the system (\ref{8.2}) with each solution $\bfx,\bfy,\bfz$ 
being counted with weight $e\left( -\Gam (z_1-z_2)\right)$. Since this weight is unimodular, 
it follows via orthogonality that 
\[
\sup_{\Gam\in [0,1)}\int_{[0,1)^3}|f(\bfalp;2X)^{10}g(\bfalp, \Gam;X)^2|\d\bfalp 
\le \int_{[0,1)^3}|f(\bfalp;2X)^{10}g(\bfalp;X)^2|\d\bfalp .
\]
The conclusion of the lemma therefore follows at once from Lemma \ref{lemma8.1}.
\end{proof}

In the interests of concision, we extract from Lemma \ref{lemma8.2} the estimate
\begin{equation}\label{8.5}
\sup_{\Gam\in [0,1)}\int_{[0,1)^3}|f(\bfalp;2X)^{10}g(\bfalp, \Gam;X)^6|\d\bfalp 
\ll_{h_2} X^{10+\eps}
\end{equation}
by means of the trivial bound $|g(\bfalp,\Gam;X)|=O(X)$. Equipped with this as a direct 
substitute for the estimate delivered by Lemma \ref{lemma3.4}, we see that no 
modification whatsoever is required in the discussion of \S\S4 to 7 in order to deliver the 
asymptotic formula (\ref{1.8}) provided that $s\ge 6$, and $X$ is sufficiently large in terms 
of $h_2$. This completes the proof of Theorem \ref{theorem1.2}. The reason that the latter 
condition concerning $h_2$ must be imposed is simply that the minor arc bound
\[
I_s(\grm(Q);X;\bfh)\ll_{h_2}X^{2s-6+\eps}Q^{-\del/4}
\]
derived in the analogue of Lemma \ref{lemma4.1} must now have dependence on $h_2$ in 
the implicit constant, as a consequence of this same dependence in (\ref{8.5}). It would not 
be difficult to ensure that the asymptotic formula (\ref{1.8}) remains valid, with an 
acceptable error term, for values of $h_2$ satisfying $|h_2|\le X^{1/2}$, or indeed a little 
larger still. However, in order to permit values of $h_2$ having absolute value nearly as 
large as $X^2$, in order to accommodate the most general situation, one would need to 
obtain sharp variants of the cap estimate (\ref{8.4}). We shall have more to say on such 
matters in a future communication.
     
\bibliographystyle{amsbracket}
\providecommand{\bysame}{\leavevmode\hbox to3em{\hrulefill}\thinspace}

\end{document}